\UseRawInputEncoding
\documentclass[leqno]{amsart}
\usepackage{stmaryrd,graphicx}
\usepackage{amssymb,mathrsfs,amsmath,amscd,amsthm,color}
\usepackage{float}
\usepackage[all,cmtip]{xy}
\DeclareMathAlphabet{\mathpzc}{OT1}{pzc}{m}{it}
\usepackage{amsfonts,latexsym,wasysym,tipa}
\newcommand{\im}{\operatorname{Im}}

\renewcommand{\int}{\operatorname{int}}

\newcommand{\ds}{\displaystyle}
\newcommand{\ui}{[0,1]}

\newcommand{\mcc}{\mathcal{C}}

\newcommand{\wild}{\mathbf{w}}

\newcommand{\scrc}{\mathscr{C}}

\newcommand{\scru}{\mathscr{U}}
\newcommand{\scrv}{\mathscr{V}}

\newcommand{\bbw}{\mathbb{W}}

\newcommand{\bbd}{\mathbb{D}}
\newcommand{\bbe}{\mathbb{E}}

\newcommand{\bbn}{\mathbb{N}}

\newcommand{\bbr}{\mathbb{R}}

\newcommand{\ov}{\overline}
\newcommand{\wt}{\widetilde}

\DeclareMathOperator{\diam}{diam}

\newtheorem{theorem}{Theorem}[section]
\newtheorem{lemma}[theorem]{Lemma}

\newtheorem{corollary}[theorem]{Corollary}
\theoremstyle{definition}\newtheorem{definition}[theorem]{Definition}

\newtheorem{remark}[theorem]{Remark}

\title{Countable models of one-dimensional Peano continua}

\author[J. Brazas]{Jeremy Brazas}
\address{West Chester University\\ Department of Mathematics\\
West Chester, PA 19383, USA}
\email{jbrazas@wcupa.edu}

\author[P. Pave\v si\' c]{Petar Pave\v si\' c}
\address{Faculty of Mathematics and Physics, University of Ljubljana and Institute of Mathematics, Physics and Mechanics, Ljubljana}
\email{petar.pavesic@fmf.uni-lj.si}

\thanks{This research was supported by the Slovenian Research Agency program P1-0292 and grant J1-4001.}

\begin{document}

\begin{abstract}
In this paper, we prove that a one-dimensional Peano continuum $X$ admits a weak homotopy equivalence $X\to C$ to a countable space $C$ if and only if the wild set $\mathbf{w}(X)$ is countable. 
\end{abstract}

\maketitle

\section{Introduction}

The homotopy theory of finite spaces was initiated by McCord \cite{mccord} and Stong \cite{stong} in 1966. It is a well-known consequence of this work that every finite topological space has the weak homotopy type of a finite polyhedron and, conversely, that every finite polyhedron is weakly homotopy equivalent to a finite space. In recent decades, finite spaces have been useful in homotopy theory by providing small (often minimal) models for weak homotopy types \cite{barmak,Barmak2,Barmark3,Chari,Forman,quillen}. More generally, any polyhedron $K$ is weakly homotopy equivalent to an $A$-space\footnote{A topological space $X$ is an \textit{$A$-space} if arbitrary intersections of open sets in $X$ are open.} $X$. Whether or not one can choose $X$ to be countable depends on whether or not the homotopy groups of $K$ are countably presented.

A \textit{Peano continuum} is a connected, locally path-connected compact metrizable space. Such spaces have cardinality of the continuum whenever they contain more than a single point. However, they need not be locally contractible and typically have uncountable homotopy groups. In \cite{brazasmatosinvolve} it is shown that there exists a countable space with an uncountable fundamental group, which appears to be a model (of weak homotopy type) for the infinite earring space $\bbe$ (see Figure \ref{fig1}). Following this example, it is not difficult to construct potential countable models for other simple one-dimensional Peano continua. However, it is far less clear if there exist Peano continua which \textit{are not} weakly homotopy equivalent to countable spaces. This brings to light the following question: \textbf{which Peano continua are weakly homotopy equivalent to a space with countably many points?}

In this paper, we provide a complete answer to this question in the case of one-dimensional Peano continua. Although some higher dimensional situations can also be analyzed, we focus our attention on one-dimensional Peano continua since it is the context where the most tools have been developed. One should not expect $A$-spaces to be useful in addressing this question since countable $A$-spaces have countable homotopy groups. 

Outside of the realm of polyhedra, various types of ``wildness" play an important role in distinguishing and classifying (weak) homotopy types. We say that a point $x$ in a metrizable space $X$ is a \textit{wild point of $X$} if every neighborhood of $x$ contains the image of a loop based at $x$ that is not null-homotopic in $X$. We let $\wild(X)$ denote the subspace of $X$ consisting of the wild points of $X$. When $X$ is a (one-dimensional) Peano continuum, it is known that $\wild(X)$ may be homeomorphic to an arbitrary (one-dimensional) compact metric space \cite{BMhigherwild}.


\begin{theorem}\label{mainthm}
Let $X$ be a one-dimensional Peano continuum. There exists a countable space $Y$ and a weak homotopy equivalence $f:X\to Y$ if and only if $\wild(X)$ is countable. 
\end{theorem}

To prove one direction of Theorem \ref{mainthm}, we show that if $f:X\to Y$ is a weak homotopy equivalence from a one-dimensional Peano continuum $X$ to any space $Y$, then $|\wild(X)|\leq|\wild(Y)|$ (see Lemma \ref{techlemma}). Thus $|\wild(X)|$ provides a lower bound for the cardinality of any space having the weak homotopy type of $X$. For the other direction, we begin with a one-dimensional Peano continuum $X$ with countable wild set $\wild(X)$. Using a result of Meilstrup and Eda \cite{Edaonedim,EdaMeilstrup,Meilstrup}, we may assume that $X\backslash \wild(X)$ is a disjoint union of open arcs. We then construct a quotient map $X\to \scrc X$ that is bijective on $\wild(X)$ and which collapses each open arc of $X\backslash\wild(X)$ to a 3-point quotient space that serves as a ``coarse open arc." The resulting space $\scrc X$ is countable, first countable, locally path-connected, and $T_0$. We complete the proof by showing the quotient map $X\to \scrc X$ is a weak homotopy equivalence (see Corollary \ref{coarsificationcor}). 

Theorem \ref{mainthm} excludes the possibility that one-dimensional Peano continua with uncountable wild sets, such as the Sierpinski Triangle, Sierpinski Carpet, or Menger Curve, have the weak homotopy type of a countable space. Additional examples from \cite{BFTestMap} include the Dyadic Arc Space $\bbd$ (see Figure \ref{fig2}) and a ``minimal" example $\mathbb{W}$ for which $\wild(\bbw)$ is homeomorphic to the ternary Cantor set (see Figure \ref{fig3}).

\begin{figure}[H]
\centering \includegraphics[height=2in]{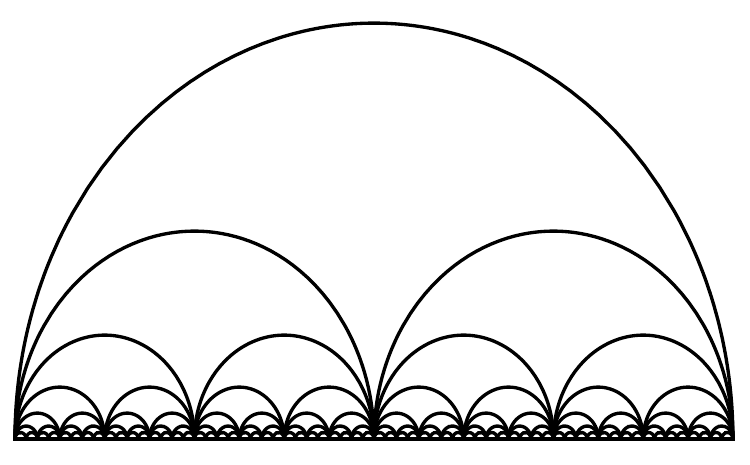}
\caption{\label{fig2} The dyadic arc-space $\mathbb{D}$ whose wild set $\wild(\mathbb{D})$ is an arc. Since $\wild(\mathbb{D})$ is uncountable, there is no countable space weakly homotopy equivalent to $\bbd$.}
\end{figure}

\begin{figure}[H]
\centering \includegraphics[height=1.1in]{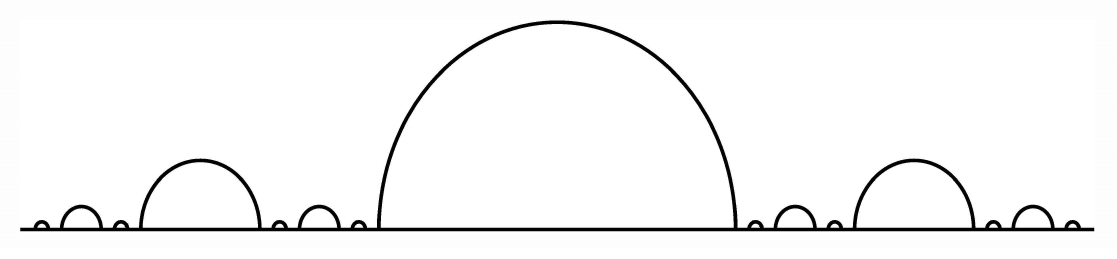}
\caption{\label{fig3} A one-dimensional Peano continuum $\mathbb{W}$ whose wild set $\wild(\mathbb{W})$ is the ternary Cantor set. Since $\wild(\mathbb{W})$ is uncountable, there is no countable space weakly homotopy equivalent to $\bbw$.}
\end{figure}

Since every countable compact metric space is homeomorphic to a countable ordinal, we can further classify the spaces to which Theorem \ref{mainthm} applies in terms of homeomorphism types of compact countable ordinals. Classifications of homeomorphism types of ordinals are known \cite{KBordinal}, including specialized classifications of homeomorphism types of countable ordinals \cite{Milliet}. 

\begin{theorem}\label{maincor}
There is a canonical bijection between homeomorphism types of compact countable non-empty ordinals and homotopy types of one-dimensional Peano continua with non-empty wild sets that admit a weak homotopy equivalence to a countable space.
\end{theorem}

\section{Wild Sets and the structure of one-dimensional Peano continua}

Generally, the term ``map" will refer to a continuous function of topological spaces. A \textit{path} in a space $X$ is a map $\alpha:\ui\to X$. We call $\alpha$ a loop if $\alpha(0)=\alpha(1)$. If $\alpha,\beta$ are paths in $X$, then $\alpha\beta$ will denote the usual concatenation and $\ov{\alpha}$ will denote the reverse of $\alpha$. Two paths $\alpha,\beta:\ui\to X$ are \textit{path-homotopic}, written $\alpha\simeq\beta$ if there exists a map $H:\ui\times\ui\to X$ such that $H(s,0)=\alpha(s)$, $H(s,1)=\beta(s)$, $H(0,t)=\alpha(0)$, and $H(1,t)=\alpha(1)$ for all $s,t\in\ui$. A map $S^n\to X$ from the $n$-sphere is said to be \textit{inessential} if it is null-homotopic and \textit{essential} otherwise.

The \textit{infinite earring space} is $\bbe=\bigcup_{k\in\bbn}\mcc_k$ where $\mcc_k\subseteq \bbr^2$ is the circle of radius $1/k$ centered at $(1/k,0)$ (see Figure \ref{fig1}). The basepoint of $\bbe$ is $b_0=(0,0)$ and we let $\ell_k:[0,1]\to \mcc_k$ denote a fixed loop traversing the $k$-th circle once counterclockwise. For each $n\in\bbn$, let $X_n=\bigcup_{k=1}^{n}\mcc_k$ be the union of the first $n$ circles. Let $r_n:\bbe\to X_n$ be the canonical retraction that maps $\bigcup_{k>n}\mcc_k$ to $b_0$ and let $r_{n+1,n}:X_{n+1}\to X_n$ be the retraction that maps $\mcc_{n+1}$ to $b_0$. Note that $\pi_1(X_n,b_0)$ is isomorphic to the group $F_n$ freely generated by $\{[\ell_k]\mid 1\leq k\leq n\}$. There is a canonical map $\phi:\pi_1(\bbe,b_0)\to \varprojlim_{k}(F_n,r_{n+1,n\#})$ to the inverse limit given by $\phi([\alpha])=([r_n\circ\alpha])_{n\in\bbn}$, which is known to be injective \cite{CConedim,EK98,MM}.

\begin{definition}
A map $f:\bbe\to X$ is \textit{fully essential} if for all $j\in\bbn$, the loop $f\circ \ell_j:S^1\to X$ is essential.
\end{definition}

\begin{definition}\label{wilddef}
Let $X$ be a space. The \textit{wild set} of $X$ is the set $\wild(X)$ of all $x\in X$ at which there exists a fully essential map $f:(\bbe,b_0)\to (X,x)$.
\end{definition}

Note that for first countable spaces, Definition \ref{wilddef} is equivalent to the definition of the wild set given in the introduction. The two definitions may differ for non-first countable spaces. However, we do not consider such spaces in this paper.

\begin{lemma}\cite{BMhigherwild}\label{closedlemma}
If $X$ is locally path-connected and first countable, then $\wild(X)$ is closed in $X$. In particular, if $X$ is a Peano continuum, then $\wild(X)$ is a compact metrizable space.
\end{lemma}

\begin{definition}
A map $f:X\to Y$ of topological spaces is
\begin{enumerate}
\item \textit{$\pi_1$-injective} if for every $x\in X$, the induced homomorphism $f_{\#}:\pi_1(X,x)\to \pi_1(Y,f(x))$ is injective.
\item a \textit{weak homotopy equivalence} if the induced function $f_{\#}:\pi_0(X)\to \pi_0(Y)$ is bijective and if for all $x\in X$ and $n\geq 1$ the induced homomorphism $f_{\#}:\pi_n(X,x)\to \pi_n(Y,f(x))$ is an isomorphism.
\end{enumerate} 
\end{definition}

\begin{lemma}\label{wildintowildlemma}
If $f:X\to Y$ is a $\pi_1$-injective map, then $f(\wild(X))\subseteq \wild(Y)$.
\end{lemma}

\begin{proof}
If $x\in \wild(X)$, there is a fully essential map $g:(\bbe,b_0)\to (X,x)$. Since $f$ is $\pi_1$-injective, $f\circ g:(\bbe,b_0)\to (Y,f(x))$ is fully essential. Thus $f(x)\in \wild(Y)$.
\end{proof}

\begin{definition}
A path $\beta:\ui\to X$ is said to be \textit{reduced} if either $\beta$ is constant or if there does not exist an interval $[a,b]\subseteq \ui$ such that $\beta(a)=\beta(b)$ and such that $\beta|_{[a,b]}:[a,b]\to X$ is null-homotopic rel. $\{a,b\}$. If $\beta$ is a reduced path that is path-homotopic to another path $\alpha$, then we say that $\beta$ is a \textit{reduced representative of} $\alpha$.
\end{definition}

In short, a reduced path is one that has no null-homotopic subloops. 

\begin{remark}\label{reducedrepremark}
Let $X$ be a one-dimensional metrizable space. We recall from \cite[Theorem 3.9]{CConedim} that every loop $\alpha:\ui\to X$ is path-homotopic to a reduced loop $\beta:\ui\to X$ by a homotopy $H:\ui^2\to X$ satisfying $\im(H)\subseteq \im(\alpha)$ \cite{CConedim}. Moreover, reduced representatives are unique up to reparameterization in the following sense: if $\beta,\beta':\ui\to X$ are reduced loops that are path-homotopic to each other, then there exists an increasing homeomorphism $h:\ui\to \ui$ such that $\beta=\beta'\circ h$.
\end{remark}

\begin{lemma}\cite[Corollary 3.3]{CConedim}\label{inclusionlemma}
If $X$ is a one-dimensional metrizable space and $A\subseteq X$, then the inclusion map $A\to X$ is $\pi_1$-injective.
\end{lemma}

\begin{proof}
Let $\alpha:\ui\to A$ be a loop based at $x_0$, which is null-homotopic in $X$. By Remark \ref{reducedrepremark}, $\alpha$ is homotopic in $X$ to its unique reduced representative, which is the constant loop at $x_0$. Moreover, this homotopy $H:\ui^2\to X$ may be chosen so that $\im(H)\subseteq \im(\alpha)$. Since $\im(H)\subseteq A$, $\alpha$ is null-homotopic in $A$. 
\end{proof}

\section{Homology of one-dimensional spaces}

In order to prove the homological analogue of Lemma \ref{inclusionlemma}, we recall Eda's notion of ``$0$-forms" \cite{Edasingularonedim16}.

\begin{definition}\label{zeroformdef}
A sequence $\beta_1,\beta_2,\dots,\beta_{2n}$ of non-constant reduced paths in a space $X$ is a \textit{$0$-form} if the concatenation $\beta_1\beta_2\cdots\beta_{2n}$ is a well-defined loop and if there exists a partition of $\{1,2,\dots,2n\}$ into pairs $\{i_1,j_1\},\{i_2,j_2\},\dots,\{i_n,j_n\}$ such that $\beta_{i_k}=\ov{\beta_{j_k}}$ for all $k\in\{1,2,\dots,n\}$.
\end{definition}

If $\alpha:\ui\to X$ is a loop in $X$, we may regard it as a $1$-cycle and let $[\alpha]_{h}$ denote the homology class of $\alpha$ in $H_1(X)$. The following lemma, due to Katsuya Eda, characterizes the structure of reduced representatives of loops with trivial homology classes.

\begin{lemma}[Zero-Form Lemma]\cite[Lemma 3.6]{Edasingularonedim16}\label{zeroformlemma}
Suppose $X$ is a one-dimensional metrizable space and $\alpha:\ui\to X$ is an essential loop based at a point $x_0\in X$. Then $[\alpha]_h=0$ if and only if there exists a $0$-form $\beta_1,\beta_2,\dots,\beta_{2n}$ such that the concatenation $\beta_1\beta_2\cdots\beta_{2n}$ is a reduced representative of $\alpha$.
\end{lemma}

\begin{lemma}\label{hinclusioncor}
If $X$ is a one-dimensional metrizable space and $A\subseteq X$, then the inclusion map $A\to X$ induces an injection $H_1(A)\to H_1(X)$ on first homology groups.
\end{lemma}

\begin{proof}
Choose $x_0\in A$ and let $i:A\to X$ denote the inclusion map. Let $\alpha:\ui\to A$ be an essential loop based at $x_0$ such that $[i\circ\alpha]_{h}=0$ in $H_1(X)$. By Lemma \ref{zeroformlemma} there exists a $0$-form $\beta_1,\beta_2,\dots,\beta_{2n}$ such that the concatenation $\beta_1\beta_2\cdots\beta_{2n}$ is a reduced representative of $i\circ\alpha$. Fix a partition of $\{1,2,\dots,2n\}$ into pairs $\{i_1,j_1\},\{i_2,j_2\},\dots,\{i_n,j_n\}$ such that $\beta_{i_k}=\ov{\beta_{j_k}}$ for all $k\in\{1,2,\dots,n\}$. Since $\beta_1\beta_2\cdots\beta_{2n}$ is a reduced representative of $\alpha$ and $\im(\alpha)\subseteq A$, by Remark \ref{reducedrepremark}, there exists a homotopy $H:\ui^2\to X$ from $\alpha$ to $\beta_1\beta_2\cdots\beta_{2n}$ with $\im(H)\subseteq \im(\alpha)$. Since $H$ has image in $A$, we have $[\alpha]=[ \beta_1\beta_2\cdots\beta_{2n}]$ in $\pi_1(A,x_0)$. Thus $[\alpha]_{h}=[\beta_1\beta_2\cdots\beta_{2n}]_{h}=\sum_{k=1}^{n}[\beta_{i_k}\ov{\beta_{j_k}}]_{h}=0$ in $H_1(A)$.
\end{proof}

\begin{lemma}\label{disjointimagehomologylemma}
Suppose $X$ is a one-dimensional path-connected metrizable space and $\alpha,\beta:\ui\to X$ are reduced loops with disjoint images. If at least one of $\alpha$ or $\beta$ is not null-homologous, then $\alpha$ and $\beta$ are not homologous.
\end{lemma}

\begin{proof}
Suppose without loss of generality that $\alpha$ is not null-homologous and suppose, to obtain a contradiction, that $\alpha$ and $\beta$ are homologous. Since $X$ is path-connected and the images of $\alpha$ and $\beta$ are closed, we may find an injective path $\gamma:\ui\to X$ with $\gamma(0)=x_0\in \im(\alpha)$, $\gamma(1)=x_1\in\im(\beta)$ and $\im(\gamma)\cap(\im(\alpha)\cup\im(\beta))=\{x_0,x_1\}$. Let $\eta:\ui\to \im(\alpha)$ be an injective path from $x_0$ to $\alpha(0)=\alpha(1)$ and let $\zeta:\ui\to \im(\beta)$ be an injective path from $x_1$ to $\beta(0)=\beta(1)$. Define the loop $L:\ui\to X$ based at $x_0$ by \[L=\eta\alpha\ov{\eta}\gamma\zeta\ov{\beta}\ov{\zeta}\ov{\gamma}.\]
Since $\alpha$ and $\beta$ are assumed to be homologous, it follows that $L$ is null-homologous. Additionally, since the paths $\alpha,\beta,\gamma,\eta,\zeta$ are all reduced, the reduced representative of $L$ has the form $L'=\alpha_1\gamma\beta_1\ov{\gamma}$ where $\alpha_1$ is the reduced representative of $\eta\alpha\ov{\eta}$ and $\beta_1$ is the reduced representative of $\zeta\beta\ov{\zeta}$. Note that $\im(\alpha_1)\subseteq \im(\alpha)$ and $\im(\beta_1)\subseteq \im(\beta)$.

By Lemma \ref{zeroformlemma}, there exists a $0$-form $\delta_1,\delta_2,\dots,\delta_{2n}$ and a partition of $\{1,2,\dots,2n\}$ into pairs $\{i_1,j_1\},\{i_2,j_2\},\dots,\{i_n,j_n\}$ such that $L'\equiv\delta_1\delta_2\cdots\delta_{2n}$ and $\delta_{i_k}=\ov{\delta_{j_k}}$ for all $k\in\{1,2,\dots,n\}$. By subdividing the product $\delta_1\delta_2\cdots\delta_{2n}$ further and adjusting the inverse pairing accordingly, we may assume that $\alpha_1\equiv\delta_1\delta_2\cdots\delta_{m}$ for some $m$. Then the image of $\delta_{m+1}\delta_{m+2}\cdots \delta_{2n}$ meets $\im(\alpha_1)$ only at the point $x_0$. Hence, whenever $i_k\in \{1,2,\dots,m\}$, both $\im(\delta_{i_k})$ and $\im(\delta_{j_k})$ are subsets of $\im(\alpha_1)$. Therefore, $\delta_1,\delta_2,\dots,\delta_{m}$ is a $0$-form and thus $\alpha_1$ is null-homologous by Lemma \ref{zeroformlemma}. Since $\alpha$ and $\alpha_1$ are homologous, it follows that $\alpha$ is null-homologous; a contradiction.
\end{proof}

\section{Standard forms for One-Dimensional Peano continua}

Next, we recall an important structural theorem. Recall that a \textit{simple closed curve} is a space homeomorphic to $S^1$.

\begin{definition}
A one-dimensional Peano continuum $X$ is in \textit{standard form} if either $X$ is a finite graph or if $\wild(X)\neq\emptyset$ and $X\backslash \wild(X)$ is a (possibly finite) disjoint union of open arcs $O_1,O_2,O_3,\dots$ such that $\ov{O_n}$ is an arc or simple closed curve with $\ov{O_n}\backslash O_n\subseteq \wild(X)$.
\end{definition}

If $X$ is in standard form, $\wild(X)\neq \emptyset$, and $X\backslash \wild(X)$ has finitely many connected components, then it is straightforward to see that $\wild(X)$ is a finite disjoint union of Peano continua $P_1,P_2,\dots,P_m$ each of which satisfies $\wild(P_j)=P_j$. Hence, if $X$ is in standard form and $\wild(X)$ has infinitely many path-components or is totally path-disconnected, then $X\backslash\wild(X)$ must have infinitely many connected components.

The following theorem appeared first in Meilstrup's Thesis \cite{Meilstrup}. However, the proof there contains a gap that was corrected later in \cite{EdaMeilstrup}. An independent proof was given by Eda in \cite{Edaonedim}. Additionally, a proof of this fact can be extracted from a ``deforestation" process studied in \cite{ConnerMeilstrupdeforest,ConnerMeilstruparcreduced}.

\begin{theorem}[Meilstrup-Eda]\label{standardformtheorem}
Every one-dimensional Peano continuum $X$ is homotopy equivalent to a one-dimensional Peano continuum that is in standard form. In particular, if $\wild(X)=\emptyset$, then $X$ is homotopy equivalent to a finite graph.
\end{theorem}

\begin{lemma}\label{odiameterlemma}
Suppose $X$ is a one-dimensional Peano continuum in standard form where $\wild(X)\neq \emptyset$ and $X\backslash \wild(X)$ is an infinite disjoint union of open arcs $\{O_n\mid n\in \bbn\}$. If $d$ is a compatible metric on $X$, then $\ds\lim_{n\to\infty}\diam(\ov{O_n})=0$.
\end{lemma}

\begin{proof}
For each $n$, fix a map $h_n:[0,1]\to \ov{O_n}$ that restricts to a homeomorphism $(0,1)\to O_n$ and such that $h_n(0)=a_n\in\wild(X)$ and $h_n(1)=e_n\in\wild(X)$. Note that $\{a_n,e_n\}$ separates $\wild(X)$ and $\ov{O_n}$. In particular, any path starting in $\wild(X)$ and ending in $\ov{O_n}$ must meet $\{a_n,e_n\}$.

Suppose, to obtain a contradiction, that there exists $n_1<n_2<n_3<\cdots$ and $\epsilon>0$ such that $\diam(\ov{O_{n_i}})\geq \epsilon$ for all $i\in\bbn$. Pick points $x_{i},y_i\in \ov{O_{n_i}}$ with $d(x_i,y_i)\geq \epsilon$. Since $X$ is compact, we may replace $\{n_i\}_{i\in\bbn}$ with a subsequence so that $\{x_i\}_{i\in\bbn}\to x$ and $\{y_i\}_{i\in\bbn}\to y$ in $X$. Since the sets $O_n$ are open, we must have $x,y\in \wild(X)$. For sufficiently large $i$, we have $y_i\notin B_d(x,\epsilon/2)$ for sufficiently large $i$. Thus $x\neq y$.

Let $\delta=\frac{d(x,y)}{2}$ and find path-connected open sets $U_x$ and $U_y$ such that $x\in U_x\subseteq B_d(x,\delta)$ and $y\in U_y\subseteq B_d(y,\delta)$. By taking cofinal subsequences, we may assume that $x_i\in U_x$ and $y_i\in U_y$ for all $i$. Any path in $U_x$ from $x$ to $x_i$ must meet $\{a_{n_i},e_{n_i}\}$ and any path in $U_y$ from $y$ to $y_i$ must also meet $\{a_{n_i},e_{n_i}\}$. If $a_{n_i}\in U_y$ for some $i$, switch the labellings of the endpoints $a_{n_i}$ and $e_{n_i}$. Thus we can arrange that $a_{n_i}\in U_x$ and $e_{n_i}\in U_y$ for all $i$. Since each set $O_{n_i}$ is connected, there must exist $z_i\in O_{n_i}\backslash (U_x\cup U_y)$. Again, replacing $\{n_i\}_{i\in\bbn}$ with a subsequence, we may assume $\{z_i\}_{i\in\bbn}\to z$ in $X$ and again we must have $z\in \wild(X)$. However, it must be that $z\notin \ov{B}_d(x,\delta/2)\cup \ov{B}_d(y,\delta/2)$. Find a path-connected open neighborhood $U_z$ of $z$ in $X$ such that $U_z$ is disjoint from $\ov{B}_d(x,\delta/2)\cup \ov{B}_d(y,\delta/2)$. For sufficiently large $i$, we have $z_i\in U_z$ and since $U_z$ is path-connected, $U_z$ must then contain either $a_{n_i}$ or $e_{n_i}$. However, this is a contradiction.
\end{proof}

\begin{corollary}\label{subsequencecor}
Suppose $X$ is a one-dimensional Peano continuum in standard form where $\wild(X)\neq \emptyset$, $X\backslash \wild(X)$ is an infinite disjoint union of open arcs $\{O_n\mid n\in \bbn\}$. Let $a_n\in \ov{O_n}\backslash O_n$ for $n\in\bbn$ and suppose $n_1<n_2<n_3<\cdots$ is such that $\{a_{n_i}\}_{i\in\bbn}\to x$ for $x\in \wild(X)$. Then for every neighborhood $U$ of $x$, we have $\ov{O_{n_i}}\subseteq U$ for all but finitely many $i\in\bbn$.
\end{corollary}

\begin{proof}
Let $d$ be a compatible metric on $X$ and find $\epsilon>0$ with $B_d(x,\epsilon)\subseteq U$. Find $i_0$ such that $a_{n_i}\in B_d(x,\epsilon/2)$ for all $i\geq i_0$. By Lemma \ref{odiameterlemma}, there exists $i_1\geq i_0$ such that $\diam(\ov{O_n})<\epsilon/2$. Thus for all $i\geq i_1$, we have $\ov{O_{n_i}}\subseteq B_d(x,\epsilon)\subseteq U$.
\end{proof}

\begin{lemma}\label{opencharlemma}
Suppose that $X$ is a one-dimensional Peano continuum in standard form where $\wild(X)\neq \emptyset$ and $X\backslash \wild(X)$ is the disjoint union of open arcs $\{O_n\mid n\in N\}$. A set $U\subseteq X$ is open if and only if
\begin{enumerate}
\item $U\cap \wild(X)$ is open in $\wild(X)$,
\item $U\cap \ov{O_n}$ is open in $\ov{O_n}$ for each $n\in\bbn$,
\item whenever $n_1<n_2<n_3<\cdots$ is such that $x_{n_i}\in \partial O_{n_i}$ and $\{x_{n_i}\}_{i\in\bbn}\to x$ in $\wild(X)$ for a point $x\in U\cap \wild(X)$, then $\ov{O_{n_i}}\subseteq U$ for all but finitely many $i$.
\end{enumerate}
\end{lemma}

\begin{proof}
Let $X'$ be the space with the same underlying set as $X$ but with the topology consisting of sets $U$ that satisfy Conditions (1)-(3) (it is straightforward to show that these form a topology). Suppose $V$ is open in $X$. Then (1) and (2) are clear and that $V$ satisfies Condition (3) follows from Corollary \ref{subsequencecor}. This shows that the identity function $id:X'\to X$ is continuous. That $id: X'\to X$ is a homeomorphism will follow once we show that $X'$ is compact.

If $U$ is open in $X'$, then $U\cap \wild(X)$ is open in $\wild(X)$ as a subspace of $X'$. By (1), $U\cap \wild(X)$ is open in $\wild(X)$ as a subspace of $X$. Thus $id: X'\to X$ maps $\wild(X)$ homeomorphically onto itself. Since $\wild(X)$ is compact as a subspace of $X$, $\wild(X)$ is compact as a subspace of $X$'.

Let $\scrv$ be an open cover of $X'$. We may assume that $\scrv$ is composed of sets satisfying the Conditions (1)-(3). Since $\wild(X)$ is compact as  a a subspace of $X'$, there exists $V_1,V_2,\dots,V_m\in\scru$ such that $\wild(X)\subseteq V=\bigcup_{i=1}^{m}V_i$. Since the sets $\ov{O_n}$ are compact, it suffices to show that $O_n\subseteq V$ for all but finitely many $n$. If, to the contrary, there exists $n_1<n_2<n_3<\cdots$ and points $y_i\in O_{n_i}\backslash V$ for all $i$. By taking a subsequence, we may assume that $\{y_i\}_{i\in\bbn}\to x$ for some $x\in X$. Since the sets $O_n$ are all open, we must have $x\in \wild(X)\subseteq V$. But then, $y_i\in V$ for sufficiently large $i$, which is a contradiction. We conclude that $\scrv$ admits a finite subcover of $X$, proving the compactness of $X'$.
\end{proof}

\section{Coarsifications of standard form Peano continua}

Our general construction extends the construction of the ``coarse earring" from \cite{brazasmatosinvolve}, which we briefly recall here. Consider the infinite earring space $\bbe$ illustrated in Figure \ref{fig1} and the countable subspace $\scrc\bbe =\{b_0\}\cup \bigcup_{n\in\bbn}\{a_n,b_n,c_n\}$. We give $\scrc\bbe$ the $T_0$ topology generated by the following basic open sets: $\{a_n\}$, $\{c_n\}$, $\{a_n,b_n,c_n\}$, $\scrc\bbe\backslash\{b_1,b_2,b_3,\dots,b_n\}$ for $n\in\bbn$.

The general idea of the construction of $\scrc\bbe$ from $\bbe$ is to replace the $n$-th open arc of $\bbe\backslash\{b_0\}$ by the three-point set $\{a_n,b_n,c_n\}$ so that $\{b_0\}\cup \{a_n,b_n,c_n\}$ is a standard finite model of $S^1$. At the same time, we must preserve the infinitary operations existing in the fundamental group of $\bbe$. This is achieved by taking the infinite sets $\scrc\bbe\backslash\{b_1,b_2,b_3,\dots,b_n\}$ to form a nested neighborhood base at $b_0$.

\begin{figure}[t]
\centering \includegraphics[height=2.3in]{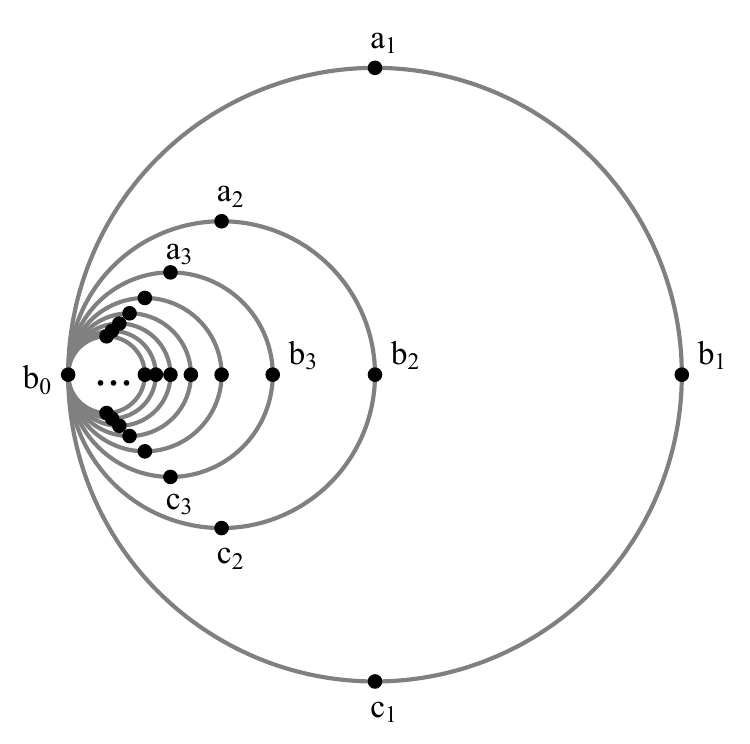}
\caption{\label{fig1} The infinite earring space $\bbe$ consisting of a shrinking infinite sequence of circles attached at a common point $b_0$. The ``coarse earring" can be constructed as the countable subset $\scrc\bbe =\{b_0\}\cup \bigcup_{n\in\bbn}\{a_n,b_n,c_n\}$ equipped with a $T_0$ topology.}
\end{figure} 

We extend this construction to a standard form one-dimensional Peano continuum $X$ with non-empty wild set. Throughout the rest of this section, we fix a one-dimensional Peano continuum $X$ in standard form and let $N$ be an initial subset of $\bbn$ so that $\{O_n\mid n\in N\}$ is an enumeration of the connected components of $X\backslash \wild(X)$. Note that $N$ may be finite or equal to $\bbn$. For each $n\in N$ fix a map $h_n:[0,1]\to \ov{O_n}$ that restricts to a homeomorphism $(0,1)\to O_n$. Define $a_n=h_n(0)$ and $e_n=h_n(1)$. Recall that $a_n=e_n$ if $\ov{O_n}$ is a simple closed curve. For each $n\in N$, set $B_n=h_n((0,1/2))$, $C_n=h_n(\{1/2\})$, and $D_n=h_n((1/2,1))$. Let $\scrc X$ be the quotient space of $X$ where each of the sets $B_n$, $C_n$, and $D_n$ is identified to a point.

Let $q:X\to \scrc X$ be the quotient map and $b_n,c_n,d_n$ be the points in $\scrc X$, which are the respective images of $B_n,C_n,D_n$. Since $\wild(X)\cup \{c_n\mid n\in N\}$ is closed in $X$ and saturated with respect to $q$, $q$ maps the subspace $\wild(X)\cup \{c_n\mid n\in N\}$ homeomorphically onto its image, which is also closed in $\scrc X$. We identify $\wild(X)\cup \{c_n\mid n\in N\}$ with its image so that $q$ is the identity on $\wild(X)\cup \{c_n\mid n\in N\}$ (see Figure \ref{fig5}). We refer to $\scrc X$ as the \textit{coarsification} of $X$ and $q:X\to \scrc X$ as the \textit{coarsification map}.

\begin{figure}[t]
\centering \includegraphics[height=1.4in]{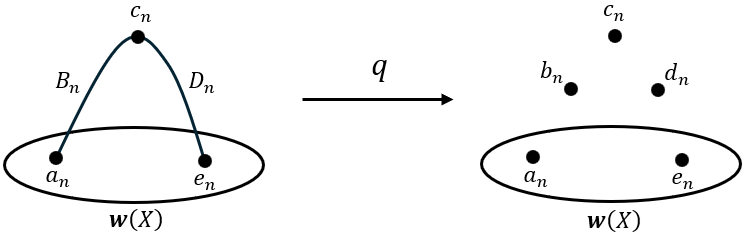}
\caption{\label{fig5} The map $q:X\to \scrc X$ collapses the open arcs $B_n$ and $D_n$ to respective points $b_n$ and $d_n$.}
\end{figure} 

\begin{remark}\label{countablespaceremark}
If $\wild(X)$ is countable, then $\scrc X=\wild(X)\cup \bigcup_{n\in N}\{b_n,c_n,d_n\}$ is countable.
\end{remark}

\begin{lemma}\label{openlemma}
The map $q:X\to \scrc X$ is an open map. 
\end{lemma}

\begin{proof}
Suppose $V$ is a non-empty open subset of $X$. Since $q$ is a quotient map, it suffices to check that $q^{-1}(q(V))$ is open in $X$. We use the characterization in Lemma \ref{opencharlemma}. First, $V\cap \wild(X)=q^{-1}(q(V))\cap \wild(X)$ is open in $\wild(X)$. For a given $n\in N$, $V\cap O_n$ is open in $O_n$. If $c_n\in V$, then $V$ meets both $B_n$ and $D_n$ giving $q^{-1}(q(V))\cap O_n=O_n$. Otherwise, $q^{-1}(q(V))\cap O_n$ is one of the sets $B_n$, $D_n$, or $B_n\cup D_n$ whenever $V\cap O_n$ is non-empty. In any of these cases, $q^{-1}(q(V))\cap O_n$ is open in $O_n$. 

Lastly, we check that $q^{-1}(q(V))$ satisfies the third condition of Lemma \ref{opencharlemma}. Suppose $n_1<n_2<n_3<\cdots$ in $N$ are such that $x_{n_i}\in \partial O_{n_i}$ and $\{x_{n_i}\}_{i\in\bbn}\to x$ in $\wild(X)$ for a point $x\in q^{-1}(q(V))\cap \wild(X)$. Since $q$ is bijective on $\wild(X)$, $x\in V$. Since the third condition from Lemma \ref{opencharlemma} applies to $V$, we have that $\ov{O_{n_i}}\subseteq V\subseteq q^{-1}(q(V))$ for all but finitely many $i$.
\end{proof}

\begin{remark}\label{boundaryremark}
For given $n\in N$, the closure of $B_n$ in $X$ is the arc $\ov{B_n}=\{a_n\}\cup B_n\cup \{c_n\}$ and the closure of $D_n$ is $\ov{D_n}=\{c_n\}\cup D_n\cup \{e_n\}$. Similarly, in $\scrc X$ the sets $\{b_n\}$ and $\{d_n\}$ are open with respective closures $\ov{\{b_n\}}=\{a_n,b_n,c_n\}$ and $\ov{\{d_n\}}=\{c_n,d_n,e_n\}$.
\end{remark}

Recall that a space $B$ is \textit{perfectly normal} if for every pair of disjoint closed sets $A_0,A_1\subseteq B$, there exists a continuous function $f:B\to \ui$ such that $f^{-1}(0)=A_0$ and $f^{-1}(1)=A_1$. It is well-known that metrizable spaces are perfectly normal.

\begin{lemma}\label{extension1}
If $B$ is perfectly normal and $T_1$, $A\subseteq B$ is a closed subspace, and $\gamma:A\to [0,1]$ is a continuous function, then there exists a continuous function $h:B\to [0,1]$ such that $h|_{A}=\gamma$ and $h(B\backslash A)\subseteq (0,1)$.
\end{lemma}

\begin{proof}
If $A=B$, set $h=\gamma$. Otherwise, fix a point $x_0\in B\backslash A$. Note that $\{x_0\}$ is closed in B. Since $B$ is perfectly normal, there exists a continuous function $f:B\to [0,1]$ such that $f^{-1}(0)=A$ and $f^{-1}(1)=\{x_0\}$. Additionally, by the Tietze Extension Theorem, there exists a map $\delta:B\to [0,1]$ such that $\delta|_{A}=\gamma$. Define $h:B\to [0,1]$ by $h(x)=\delta(x)(1-f(x))+\frac{1}{2}f(x)$. Then $h$ is continuous and satisfies $h|_{A}=\delta|_{A}=\gamma$. If $x\in B\backslash A$, then $f(x)>0$ and it follows that $0<h(x)<1$. Thus $h(B\backslash A)\subseteq (0,1)$.
\end{proof}

\begin{lemma}\label{liftinglemma}
Let $q:X\to \scrc X$ be the coarsification map, $B$ be a locally path-connected metrizable space, and $A\subseteq B$ be a closed subspace. If $f:A\to X$ and $g:B\to \scrc X$ are maps such that $q\circ f=g|_{A}$, then there exists a continuous lift $\wt{g}:B\to X$ such that the following diagram commutes.
\[\xymatrix{
A \ar[r]^-{f} \ar@{_{(}->}[d] & X \ar[d]^-{q} \\
B \ar@{-->}[ur]^-{\wt{g}} \ar[r]_-{g} & \scrc X
}\]
\end{lemma}

\begin{proof}
Recall that $q$ is a homeomorphism on the closed set $\wild(X)\cup\{c_n\mid n\in N\}$. Let $Z_0=g^{-1}(\wild(X)\cup\{c_n\mid n\in N\})$. Define $k:A\cup Z_0\to X$ by $k(b)=q^{-1}(g(b))$ for $b\in Z_0$ and $k(a)=f(a)$ for $a\in A$. If $z\in A\cap Z_0$, then $q(f(z))=g(z)$ and thus $f(z)=q^{-1}(g(z))$. Hence, $k$ is well-defined and continuous by the Pasting Lemma. 

For each $n\in N$, let $U(b_n)=g^{-1}(b_n)\backslash A$ and $U(d_n)=g^{-1}(d_n)\backslash A$. Note that $B\backslash (A\cup Z_0)$ is the disjoint union of the open sets $U(b_n)$ and $U(d_n)$ for $n\in N$. It follows that for each $n\in N$, we have $\partial U(b_n)\subseteq A\cup Z_0$. Since $\ov{b_n}=\{a_n,b_n,c_n\}$ and $g(U(b_n))=b_n$, we have $g(\partial U(b_n))\subseteq \{a_n,b_n,c_n\}$ for each $n\in N$. Since $q^{-1}(\{a_n,b_n,c_n\})=\ov{B_n}$ and $q\circ k=g|_{A\cup Z_0}$, we have $k(\partial U(b_n))\subseteq \ov{B_n}$. Define $\gamma_n=k|_{\partial U(b_n)}:\partial U(b_n)\to \ov{B_n}$. Since $\ov{B_n}$ is an arc (which may be identified with $\ui$) in $X$ with endpoints $a_n$ and $c_n$ and $\partial U(b_n)$ is a closed subspace of the perfectly normal space $\ov{U(b_n)}$, Lemma \ref{extension1} guarantees the existence of a map $h_n:\ov{U(b_n)}\to \ov{B_n}$ such that $(h_n)|_{\partial U(b_n)}=\gamma_n$ and such that $h_n(U(b_n))\subseteq B_n$.

Similarly, if we define $\delta_n=k|_{\partial U(d_n)}:\partial U(d_n)\to \ov{D_n}$, the same argument gives the existence of a map $k_{n}:\ov{U(d_n)}\to \ov{D_n}$ such that $(k_n)|_{\partial U(d_n)}=\delta_n$ and $k_n(U(d_n))\subseteq D_n$. 

Define $\wt{g}:B\to X$ so that 
\[\wt{g}(b)= \begin{cases}
k(b) & \text{ if }b\in A\cup Z_0\\
h_n(b) & \text{ if }b\in U(b_n)\\
k_n(b) & \text{ if }b\in U(d_n)
\end{cases}\]
Certainly, $\wt{g}$ is well-defined as a function and makes the diagram in the statement of the lemma commute. Since $k$ is continuous, $\wt{g}$ is continuous on $A\cup Z_0$. For each $n\in N$, $\wt{g}$ is also continuous on $\ov{U(b_n)}$ and $\ov{U(d_n)}$ by the continuity of $h_n$ and $k_n$ respectively. If $N$ is finite, then $\wt{g}$ is continuous by the Pasting Lemma so we assume that $N$ is infinite. To complete the proof of continuity, it suffices to consider the following case: Suppose $n_1<n_2<n_3<\cdots$ are integers in $N$ and $\{s_k\}_{k\in\bbn}\to s$ is a convergent sequence in $B$ where $s_k\in U(b_{n_k})\cup U(d_{n_k})$ for each $k\in\bbn$. Since the sets $U(b_n)\cup U(d_n)$, $n\in N$ are open and disjoint in $B$, we must have $s\in A\cup Z_0$. Note that $g(s_k)\in\{b_{n_k},d_{n_k}\}$ for some $n_k\in N$ and $\{g(s_k)\}_{k\in\bbn}\to g(s)$ in $\scrc X$ (by the continuity of $g$) where $g(s)\in\wild(X)\cup\{c_n\mid n\in N\}$. By our construction of $\wt{g}$, we have $\wt{g}(s_k)\in B_{n_k}\cup D_{n_k}\subseteq O_{n_k}$ for each $k\in\bbn$.

If $g(s)=c_m$ for $m\in N$, then $\wt{g}(s)=c_m$ and the sequence $\{n_k\}_{k\in\bbn}$ of integers is eventually constant at some $m$. Thus there exists $k_0$ such that $g(s_k)\in \{b_m,d_m\}$ for all $k\geq k_0$. Thus $s_k\in U(b_m)\cup U(d_m)$ for all $k\geq k_0$. By the Pasting Lemma and continuity of $h_m$ and $k_m$, $\wt{g}$ is continuous on $\ov{U(b_m)}\cup \ov{U(d_m)}$. Thus $\{\wt{g}(s_k)\}_{k\in\bbn}\to \wt{g}(s)$.

Otherwise, $g(s)=x\in \wild(X)$. Let $V$ be path-connected open neighborhood of $x$ in $X$. By Lemma \ref{openlemma}, $q(V)$ is open in $\scrc X$. Thus there exists $k_0$ such that $g(s_k)\in q(V)$ for all $k\geq k_0$. Since $g(s_k)\in \{b_{n_k},d_{n_k}\}$ and the set $\{a_{n_k},e_{n_k}\}$ separates $\scrc X$, any path in $\scrc X$ starting at $x$ and ending in $\{b_{n_k},d_{n_k}\}$ must have image that meets $\{a_{n_k},e_{n_k}\}$. Since $q(V)$ is a path-connected open neighborhood of $x$, for each $k\geq k_0$ we must therefore have $q(V)\cap \{a_{n_k},e_{n_k}\}\neq \emptyset$. But since $ V\cap \wild(X)=q(V)\cap \wild(X)$, we conclude that $V\cap \{a_{n_k},e_{n_k}\}=V\cap \partial O_{n_k}\neq\emptyset$ for all $k\geq k_0$. But Condition (3) of Lemma \ref{opencharlemma} gives that $\ov{O_{n_k}}\subseteq V$ for all but finitely many $k\geq k_0$. Thus $\wt{g}(s_k)\in  B_{n_k}\cup D_{n_k}\subseteq \ov{O_{n_k}}\subseteq V$ for all but finitely many $k$.
\end{proof}

Applying Lemma \ref{liftinglemma} to the pairs $(B,A)=(I^{n+1},I^n\times\{0\})$ and $(B,A)=(I^{n+1},\partial I^{n+1})$ for all $n\geq 0$, we have the following.

\begin{corollary}\label{coarsificationcor}
For any standard form one-dimensional Peano continuum $X$ with non-empty wild set, the coarsification map $q_{\bbe}:X\to \scrc X$ is a Serre fibration and a weak homotopy equivalence.
\end{corollary}


\section{Proofs of Main Results}

To complete the proofs of our main results, we first recall another theorem of Eda.

\begin{theorem}\cite[Theorem 1.1]{EdaSpatial}\label{edaspatialthm}
If $X$ is a one-dimensional Peano continuum and $\psi:\pi_1(\bbe,b_0)\to \pi_1(X,x_0)$ is a homomorphism, then there exists a map $g:\bbe\to X$ and a path $\alpha:\ui\to X$ from $\alpha(0)=g(b_0)$ to $\alpha(1)=x_0$ such that $\varphi_{\alpha}\circ g_{\#}=\psi$ where $\varphi_{\alpha}:\pi_1(X,g(b_0))\to \pi_1(X,x_0)$, $\varphi_{\alpha}([\gamma])=[\ov{\alpha}\gamma\alpha]$ is the basepoint-change isomorphism.
\end{theorem}

We note that there is a uniqueness part to Theorem \ref{edaspatialthm}. However, we extract from this a more specialized statement that applies directly to our situation.

\begin{lemma}\cite[Prop. 9.13]{BFpants}\label{noslidinglemma}
If $X$ is a one-dimensional metric space, $f_1,f_2:\bbe\to X$ are fully essential maps, and $\alpha:\ui\to X$ is a path from $f_2(b_0)$ to $f_1(b_0)$ such that $f_1\circ\ell_j\simeq \ov{\alpha}(f_2\circ \ell_j)\alpha$ for all $j\in\bbn$, then $\alpha(0)=\alpha(1)$.
\end{lemma}

\begin{remark}\label{scremark}
It is known that a path-connected one-dimensional metrizable space $X$ is simply connected if and only if it contains no simple closed curve. For if $\alpha:S^1\to X$ is an embedding, then the inclusion $\im(\alpha)\to X$ is $\pi_1$-injective by Lemma \ref{inclusionlemma}, showing that $X$ is not simply connected. Conversely, if $X$ does not contain a simple closed curve, the image of a given loop $\alpha:S^1\to X$ is a dendrite, i.e. a unique arcwise connected Peano continuum. Since dendrites are contractible \cite[Lemma 2]{Fort}, $\alpha$ is null-homotopic.
\end{remark}

\begin{lemma}\label{choosehomlemma}
If $X$ is a one-dimensional metrizable space and $x\in \wild(X)$, then there exists a map $f:(\bbe,b_0)\to (X,x)$ such that $[f\circ \ell_n]_{h}\neq 0$ in $H_1(X)$ for all $n\in\bbn$.
\end{lemma}

\begin{proof}
Let $U_1\supseteq U_2\supseteq U_3\supseteq $ be a neighborhood base of path-connected open sets at $x$. Since $x\in \wild(X)$, Remark \ref{scremark} implies that the open set $U_n$ must contain a simple closed curve $C_n$. Let $\gamma_n:\ui\to C_n$ be a loop based at $c_n=\gamma_n(0)$ that parameterizes $C_n$, i.e. so that $\gamma_{n}^{-1}(c_n)=\{0,1\}$ and $|\gamma_{n}^{-1}(c)|=1$ for $c\neq c_n$. If $x\notin C_n$, find an injective path $\eta_n:\ui\to U_n$ from $x$ to a point $c_n=\eta_n(1)\in C_n$ so that $\im(\eta_n)\cap C_n=\{c_n\}$. In the case that $x\in C_n$, we set $c_n=x$ and take $\eta_n$ to be the constant path at $c_n$. Define $f:(\bbe,b_0)\to (X,x)$ so that $f\circ\ell_n=\eta_n\gamma_n\ov{\eta_{n}}$ for each $n\in\bbn$. With $f$ constructed, we fix $n\in\bbn$. Let $i:C_n\to X$ be the inclusion map and recall from Lemma \ref{hinclusioncor} that the induced homomorphism $i_{\ast}:H_1(C_n)\to H_1(X)$ is injective. Since $[\gamma_n]\neq 0$ in $H_1(C_n)$, $[i\circ \gamma_n]_{h}\neq 0$ in $H_1(X)$. Since $i\circ \gamma_n$ and $f\circ \ell_n$ are homologous in $X$, we have $[f\circ \ell_n]_{h}\neq 0$ in $H_1(X)$.
\end{proof}

The Hurewicz Theorem (restricted to dimension one) states that for a path-connected space $Y$ and point $y_0\in Y$, the kernel of the Hurewicz homomorphism $\pi_1(Y,y_0)\to H_1(Y)$, $[\alpha]\mapsto [\alpha]_h$ is precisely the commutator subgroup of $\pi_1(Y,y_0)$. Hence, $H_1(Y)$ is isomorphic to the abelianization of $\pi_1(Y,y_0)$. Since we plan to apply this statement in a case where $Y$ is not assumed to be Hausdorff, we find it prudent to highlight that the standard proof of this result \cite[Theorem 2A.1]{Hatcher} does not require any assumptions on $Y$ except that $Y$ be path connected.

\begin{lemma}\label{techlemma}
Let $(X,d)$ be a path-connected, locally path-connected, one-dimensional metric space and $Y$ be any path-connected space. If $f:X\to Y$ is a map that induces an isomorphism on fundamental groups, then the restriction $f|_{\wild(X)}:\wild(X)\to \wild(Y)$ is an injective function.
\end{lemma}

\begin{proof}
Since $f$ is $\pi_1$-injective, we have $f(\wild(X))\subseteq \wild(Y)$ by Lemma \ref{wildintowildlemma} and thus the restriction $f|_{\wild(X)}:\wild(X)\to \wild(Y)$ is a well-defined continuous function.

Suppose to obtain a contradiction that there exists distinct points $x_0,x_1\in \wild(X)$ such that $f(x_0)=y_0=f(x_1)$. We will take $x_0$ to be the basepoint of $X$. Since $x_0,x_1\in \wild(X)$, there exist fully essential maps $g_i:(\bbe,b_0)\to (X,x_i)$, $i\in\{0,1\}$. We choose $g_1:(\bbe,b_0)\to (X,x_1)$ arbitrarily. Applying Lemma \ref{choosehomlemma}, we construct $g_0:(\bbe,b_0)\to (X,x_0)$ as in the proof of Lemma \ref{choosehomlemma} so that $[g_0\circ \ell_n]_{h}\neq 0$ in $H_1(X)$ for all $n\in\bbn$.

Since $f$ is assumed to induce an isomorphism on fundamental groups, the maps $f\circ g_i:\bbe\to Y$ are both fully essential and based at $y_0$. Thus we may define a fully essential map $h:(\bbe,b_0)\to (Y,y_0)$ by $h\circ\ell_{2j-1}=f\circ g_0\circ \ell_j$ and $h\circ\ell_{2j}=f\circ g_1\circ\ell_j$ for all $j\in\bbn$ (see Figure \ref{fig4}). 

\begin{figure}[h]
\centering \includegraphics[height=2.3in]{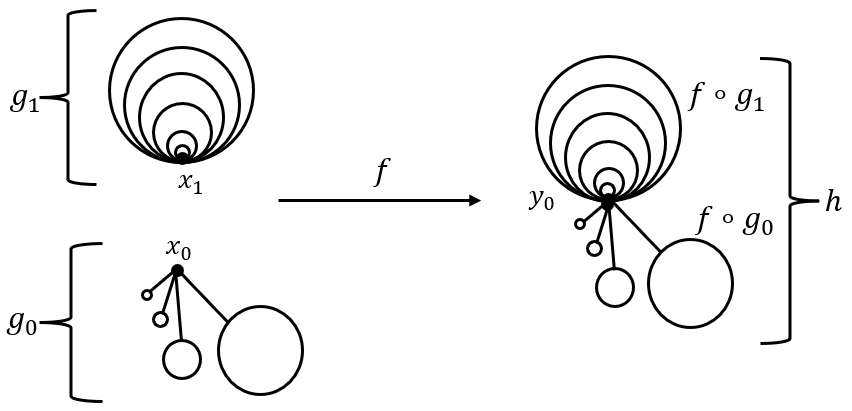}
\caption{\label{fig4} The maps $g_i:(\bbe,b_0)\to (X,x_i)$, $i\in\{0,1\}$ (where $g_0$ is chosen to have a special structure) are composed with $f$ and combined into a single map $h:\bbe\to Y$. }
\end{figure}

Consider the homomorphism $f_{\#}^{-1}\circ h_{\#}:\pi_1(\bbe,b_0)\to \pi_1(X,x_1)$ where $f_{\#}$ is induced by the based map $(X,x_1)\to (Y,y_0)$. By Theorem \ref{edaspatialthm}, there exists a map $g:\bbe\to X$ and a path $\delta:\ui\to X$ from a point $x_2=g(b_0)=\delta(0)$ to $\delta(1)=x_1$ such that $\varphi_{\delta}\circ g_{\#}=f_{\#}^{-1}\circ h_{\#}$.
\[\xymatrix{
\pi_1(X,x_1) \ar[r]^-{f_{\#}}_-{\cong} & \pi_1(Y,y_0) \\
\pi_1(X,x_2) \ar[u]^-{\varphi_{\delta}}_-{\cong} & \pi_1(\bbe,b_0) \ar[l]^-{g_{\#}} \ar[u]_-{h_{\#}}
}\]
For all $j\in\bbn$, we have $f_{\#}([\ov{\delta} (g\circ\ell_j) \delta])=[h\circ \ell_j]$ in $\pi_1(Y,y_0)$. In particular, $f_{\#}([\ov{\delta}(g\circ\ell_{2j})\delta])=[h\circ \ell_{2j}]=f_{\#}([g_1\circ \ell_j])$ in $\pi_1(Y,y_0)$ for each $j\in\bbn$. Since $f_{\#}:\pi_1(X,x_1)\to \pi_1(Y,y_0)$ is injective, we have $\ov{\delta} (g\circ\ell_{2j})\delta\simeq g_1\circ \ell_j$. Let $g_2:\bbe\to X$ be the map given by $g_2\circ \ell_j=g\circ\ell_{2j}$ for each $j\in\bbn$. Applying Lemma \ref{noslidinglemma} to the maps $g_1$ and $g_2$, we conclude that $x_2=g_2(b_0)= g_1(b_0)=x_1$. 

Define $g_3:(\bbe,b_0)\to (X,x_1)$ to satisfy $g_3\circ\ell_j=g\circ \ell_{2j-1}$ for all $j\in\bbn$. Since $x_0\neq x_1$, we may choose $n$ sufficiently large so that $\im(g_0\circ \ell_n)\cap \im(g_3\circ \ell_n)=\emptyset$. Set $a=f_{\#}([\ov{\delta} (g_3\circ\ell_{n})\delta])=[h\circ \ell_{2n-1}]\in \pi_1(Y,y_0)$ and note that $a=[f\circ g_0\circ \ell_n]$. Fix a path $\alpha:\ui\to X$ from $x_0$ to $x_1$ and set $b=[f\circ\alpha]\in \pi_1(Y,y_0)$. Define $L=\alpha(\ov{\delta}(g_3\circ\ell_{n})\delta)\ov{\alpha}(\ov{g_0\circ \ell_n})$ and note that $L$ is a loop in $X$ based at $x_0$. Additionally, $[f\circ L]=bab^{-1}a^{-1}$ is a commutator element in $\pi_1(Y,y_0)$. Thus $[f\circ L]_{h}=0$ in $H_1(Y)$. Since $f$ induces an isomorphism on fundamental groups, the Hurewicz Theorem and the naturality of the Hurewicz homomorphism imply that $f$ also induces an isomorphism $f_{\ast}:H_1(X)\to H_1(Y)$ on first homology groups. Thus $[L]=0$ in $H_1(X)$.

By canceling the inverse pairs in the definition of $L$, we see that $[L]_{h}=[g_3\circ\ell_{n}]_{h}-[g_0\circ\ell_{n}]_{h}$. Thus $[g_3\circ \ell_n]_h=[g_0\circ \ell_n]_h$ in $H_1(X)$. However, this is a contradiction to Lemma \ref{disjointimagehomologylemma} since $g_0\circ \ell_n$ and $g_3\circ \ell_n$ have disjoint images and $g_0\circ \ell_n$ was chosen so that $[g_0\circ \ell_n]\neq 0$.
\end{proof}

\begin{remark}
The restriction map $f|_{\wild(X)}:\wild(X)\to \wild(Y)$ in Lemma \ref{techlemma} need not be bijective. Indeed, the map $f:\bbe\to \bbe\times\ui$, $f(x)=(x,0)$ is a homotopy equivalence but $\wild(\bbe\times\ui)=\{b_0\}\times\ui$ is uncountable.
\end{remark}




Finally, we give a proof of our main result. Recall the statement of Theorem \ref{mainthm} from the introduction.

\begin{proof}[Proof of Theorem \ref{mainthm}]
Let $X$ be a one-dimensional Peano continuum. First we note that in the case where $\wild(X)=\emptyset$, Theorem \ref{standardformtheorem} implies that $X$ is homotopy equivalent to a finite graph. Then $X$ is weakly homotopy equivalent to a finite space by \cite[Theorem 3]{mccord}. With this case completed, we assume that $\wild(X)\neq \emptyset$. 

First, suppose there exists a countable space $Y$ and a weak homotopy equivalence $f:X\to Y$. Since $f$ induces an isomorphism on fundamental groups, the restriction $f|_{\wild(X)}:\wild(X)\to\wild(Y)$ is a well-defined injective function by Lemma \ref{techlemma}. Since $Y$ is countable, $\wild(Y)$ is countable and since $\wild(X)$ injects into $\wild(Y)$, $\wild(X)$ is countable.

For the converse, suppose $\wild(X)$ is countable. By Theorem \ref{standardformtheorem}, we may find a one-dimensional Peano continuum $X'$ that is in standard form and a homotopy equivalence $g:X\to X'$. Let $h:X'\to X$ be a choice of homotopy inverse. By Lemma \ref{techlemma}, $h$ maps $\wild(X')$ injectively into $\wild(X)$. Thus $\wild(X')$ is countable. Hence, by replacing $X$ with $X'$, we may assume that $X$ is in standard form. By Corollary \ref{coarsificationcor}, the coarsification map $q:X\to \scrc X$ is a weak homotopy equivalence. Since $\scrc X$ is a countable set whenever $\wild(X)$ is countable (see Remark \ref{countablespaceremark}), we set $Y=\scrc X$ and $f=q$ to complete the proof.
\end{proof}

To prove our second main result, we recall another recent result of Katsuya Eda.

\begin{theorem}\cite{Edazerodim}\label{edazerodimtheorem}
Let $X$ and $Y$ be one-dimensional Peano continua for which $\wild(X)$ and $\wild(Y)$ are zero-dimensional. Then $X\simeq Y$ if and only if $\wild(X)\cong \wild(Y)$.
\end{theorem}

We conclude with a proof of our second main result. Recall the statement of Theorem \ref{maincor} from the introduction. We regard an ordinal $\lambda$ as the set of its predecessors and endow $\lambda$ with the natural order topology. We write $Ho(X)$ to denote the homotopy type of a topological space $X$.

\begin{proof}[Proof of Theorem \ref{maincor}]
Given a compact countable ordinal $\lambda\geq 1$, fix an order-preserving embedding $f:\lambda\to \ui$ with $f(0)=0$ and $f(\max(\lambda))=1$ and identify $\lambda$ with its image in $\ui$ under $f$. Fix an injection $g:\lambda\to \bbn$.

For each $\kappa\in\lambda$, let $E_{\kappa}\subseteq \ui\times [0,\infty)$ be a homeomorphic copy of $\bbe$ satisfying $\wild(E_{\kappa})=\{(f(\kappa),0)\}$ and $diam(E_{\kappa})<2^{-g(\kappa)}$. Moreover, we choose the countable collection of spaces $E_{\kappa}$, $\kappa\in \lambda$ to be pairwise-disjoint. Define $A_{\lambda}=\ui\times\{0\}\cup \bigcup_{\kappa\in \lambda}E_{\kappa}$. Note that $A_{\lambda}$ is a one-dimensional Peano continuum in standard form with countable, compact, and zero-dimensional wild set $\wild(A_{\lambda})=\lambda$. If $\lambda '$ is another ordinal homeomorphic to $\lambda$, then $\wild(A_{\lambda '})=\lambda '\cong\lambda= \wild(A_{\lambda})$ and thus $A_{\lambda}\simeq A_{\lambda'}$ by Theorem \ref{edazerodimtheorem}. Hence, there is a well-defined function making the assignment $[\lambda]\to Ho(A_{\lambda})$ of the homeomorphism class $[\lambda]$ of a compact countable ordinal $\lambda$ to the homotopy type $Ho(A_{\lambda})$. We claim that this function is a bijection.

If $\lambda,\lambda'$ are non-empty compact countable ordinals and $Ho(A_{\lambda})=Ho(A_{\lambda'})$, then there is a homotopy equivalence $h:A_{\lambda}\to A_{\lambda'}$. It is known that a homotopy equivalence of one-dimensional Peano continua restricts to a homeomorphism on wild sets \cite[Theorem 9.15]{BFpants}. Hence, $\lambda=\wild(A_{\lambda})\cong \wild(A_{\lambda'})=\lambda'$. This proves injectivity.

To complete the proof, we must show that the image of the function $[\lambda]\to Ho(A_{\lambda})$ is precisely the set of homotopy types of one-dimensional Peano continua that are weakly homotopy equivalent to a countable space. For given $\lambda\geq 1$, $A_{\lambda}$ is a one-dimensional Peano continuum with countable wild set $\wild(A_{\lambda})$. It follows from Theorem \ref{mainthm} that $A_{\lambda}$ is weakly homotopy equivalent to a countable space. On the other hand, suppose $X$ is a one-dimensional Peano continuum with countable non-empty wild set. Then $\wild(X)$ is a compact countable metrizable space. It is well-known that every compact countable metrizable space is homeomorphic to a countable ordinal. Hence, $\wild(X)\cong \lambda$ for a countable compact ordinal $\lambda\geq 1$. Since $A_{\lambda}$ is a one-dimensional Peano continuum with zero-dimensional wild set $\lambda=\wild(A_{\lambda})$ and $\wild(A_{\lambda})\cong \wild(X)$, we have $Ho(X)=Ho(A_{\lambda})$ by Theorem \ref{edazerodimtheorem}.
\end{proof}

\end{document}